\documentclass{amsart}
\usepackage{amsthm}
\usepackage{amsmath,amssymb}
\newtheorem{thm}{Theorem}[section]

\newtheorem{lem}[thm]{Lemma}

\theoremstyle{definition}

\theoremstyle{remark}
\newtheorem{rem}[thm]{Remark}

\numberwithin{equation}{section}

\newcommand{\R}{\mathbb{R}}

\begin{document}

\title[ ]{Sturm--Picone theorem for fractional nonlocal equations}
\author[J.\,Tyagi ]
{ J.Tyagi }

\address{J.\,Tyagi \hfill\break
 Discipline of Mathematics \newline
 Indian Institute of Technology Gandhinagar \newline
 Palaj, Gandhinagar \newline
 Gujarat, India - 382355}
 \email{jtyagi@iitgn.ac.in, jtyagi1@gmail.com}

\thanks{Submitted 01--11--2018.  Published-----.}
\begin{abstract}
In this paper, we establish a generalization of  Sturm--Picone comparison theorem for a pair of fractional nonlocal  equations:
\begin{eqnarray*}
 \begin{gathered}
 (-div. (A_1(x)\nabla))^{s} u =  C_{1}(x) u \,\,\,\mbox{in}\,\,\Omega,\\
 u   = 0 \,\,\,\,\mbox{on}\,\,\,\,\,\,\,\partial \Omega,
\end{gathered}
 \end{eqnarray*}
and
 \begin{eqnarray*}
 \begin{gathered}
 (-div. (A_2(x)\nabla))^{s} v =  C_{2}(x) v \,\,\,\mbox{in}\,\,\Omega,\\
 v   = 0 \,\,\,\,\mbox{on}\,\,\,\,\,\,\,\partial \Omega,
\end{gathered}
 \end{eqnarray*}
where $\Omega\subset \R^n$ is an open bounded subset with smooth boundary, $0<s<1,\,\,A_1,\,A_2$ are real symmetric and positive definite matrices
on $\Omega$ with continuous entries on $\overline{\Omega}$ and $C_{1}, C_{2}\in C(\overline{\Omega}).$ 
\end{abstract}
\maketitle
\section{Introduction}
In this paper, we are interested to generalize Sturm--Picone comparison theorem for a pair of fractional nonlocal  equations:
\begin{eqnarray}\label{no1}
 \begin{gathered}
 (-div. (A_1(x)\nabla))^{s} u =  C_{1}(x) u \,\,\,\mbox{in}\,\,\Omega,\\
 u   = 0 \,\,\,\,\mbox{on}\,\,\,\,\,\,\,\partial \Omega,
\end{gathered}
 \end{eqnarray}
and
 \begin{eqnarray}\label{no2}
 \begin{gathered}
 (-div. (A_2(x)\nabla))^{s} v =  C_{2}(x) v \,\,\,\mbox{in}\,\,\Omega,\\
 v   = 0 \,\,\,\,\mbox{on}\,\,\,\,\,\,\,\partial \Omega,
\end{gathered}
 \end{eqnarray}
where $\Omega\subset \R^n$ is an open bounded subset with smooth boundary, $0<s<1,\,\,A_1,\,A_2$ are real symmetric and positive definite matrices
on $\Omega$ with continuous entries on $\overline{\Omega}$ and $C_{1}, C_{2}\in C(\overline{\Omega}).$ The nonlocal fractional operator $(-div. (A(x)\nabla))^{s} u,$ 
where $A$ is a real symmetric matrix is defined next.

Let us recall briefly the earlier developments on this subject which have played important roles in the qualitative theory of differential equations.
In 1836, Sturm\,\cite{sturm} established the first important comparison theorem which deals with a pair of linear ODEs
\begin{eqnarray}
lx\equiv(p_1(t)x'(t))'+ q_1(t)x(t)=0.\label{1}\,\,\\
Ly\equiv(p_2(t)y'(t))'+ q_2(t)y(t)=0,\label{2}\,
\end{eqnarray}
on a bounded interval $(t_1,\,t_2),$ where $p_1,\,p_2,\,q_1,\,q_2$ are real-valued continuous functions and 
$p_1(t)>0,\,p_2(t)>0$ on $[t_1,\,t_2]\subset (0, \infty).$ The original Sturm's comparison theorem \cite{sturm} reads as 
\begin{thm}\rm{(Sturm's comparison theorem)}\label{th.1}
Suppose $p_1(t)=p_2(t)$ and $q_1(t)>q_2(t),\,\forall\,t\in(t_1,\,t_2).$ If there exists a nontrivial real solution $y$ of 
\eqref{2} such that $y(t_1)=0=y(t_2),$ then every real solution of \eqref{1}  has at least  one zero in $(t_1,\,t_2).$ 
\end{thm}
In 1909, Picone\,\cite{picone} modified Sturm's theorem. The modification reads as 
\begin{thm}\rm{(Sturm--Picone theorem)}\label{th.2}
Suppose that $p_2(t)\geq p_1(t)$ and $q_1(t)\geq q_2(t),\,\forall\,t\in(t_1,\,t_2).$ If there exists a nontrivial real 
solution $y$ of \eqref{2} such that $y(t_1)=0=y(t_2),$ then every real solution of \eqref{1} unless a constant multiple of $y$ has at 
least  one zero in $(t_1,\,t_2).$ 
\end{thm}
In 1962, Leighton\,\cite{leig} proved a comparison theorem to the above pair of Equations \eqref{1}--\eqref{2}. He showed that Sturm 
and  Sturm-Picone theorems may be regarded as special cases of this theorem. 
In order to prove his theorem, he defined the quadratic functionals associated with \eqref{1} and \eqref{2} as follows:
\begin{eqnarray}
j(u)=\int_{t_1}^{t_2} [p_1(t)(u'(t))^2 -  q_{1}(t) (u(t))^2] dt.\nonumber\\
J(u)=\int_{t_1}^{t_2} [p_2(t)(u'(t))^2 -  q_{2}(t)  (u(t))^2] dt,\nonumber
\end{eqnarray}
where the domain $D$ of $j$ and $J$ is defined to be the set of all real-valued functions $u \in C^1[t_1,\,t_2]$ such 
that $u(t_1)= u(t_2)=0\,\,(t_1,\,t_2\,\mbox{are consecutive zeros of}\, u).$
The variation of $j(u)$ is defined as $V(u)= J(u)-j(u),\,\,$ i.e.,
\begin{equation}
V(u)= \int_{t_1}^{t_2} [(p_2(t)-p_{1}(t))(u'(t))^2 + ( q_1(t)- q_{2}(t) ) (u(t))^2] dt.\label{eq3}
\end{equation}
Now, Leighton's theorem reads as follows:
\begin{thm}\rm{(Leighton's theorem)}\label{th.3}
Suppose there exists a nontrivial real solution $u$ of $Lu=0$ in $(t_1, t_2)$ such that $u(t_1)=u(t_2)=0$ and $V(u)\geq 0,$ then every real solution of
$lv=0$ unless a constant multiple of $u$ has at 
least  one zero in $(t_1,\,t_2).$ 
\end{thm}
It is easy to see that Theorems\,\ref{th.1}\,and \,\ref{th.2} are special cases of Leighton's theorem. We point out that the proof of Leighton's theorem heavily depends on a lemma
so-called Leighton's variational lemma, which is stated as follows:
\begin{lem}\rm{(Leighton's variational lemma)}\label{th.4}
If there exists a function $u\in D,$ not identically zero, such that  $J(u)\leq 0,$ then every real solution of 
$Lv=0$  except a constant multiple of $u$ vanishes at some point of $(t_1, t_2).$  
\end{lem}

We refer to a very recent work \cite{ghat}, where the authors consider a pair of equations of the form
$$−(p(u'+su))'+rp(u'+su)+qu=0 $$
on a finite interval, where $1/p, r, s$ and $q$ are real integrable functions. They established a generalization of Leighton's  comparison theorem for these equations and as  
special cases, they provide a generalization of a Sturm-Picone-type theorem and a generalization of a Sturm-type separation theorem.

We mention that most of the above comparison theorems have been extended to a pair of linear elliptic partial differential equations of type
\begin{eqnarray}
lu\equiv  \sum_{i, j=1}^{n} D_{i}(a_{ij} D_{j}u) + c u=0.\label{el1}\,\,\\
Lv\equiv   \sum_{i, j=1}^{n} D_{i}(A_{ij} D_{j}v) + C v=0,\label{el2}\,
\end{eqnarray}
in $\Omega\subset \R^{n},$ where $\Omega$ is a bounded domain with smooth boundary, $a_{ij}, A_{ij}, c, C$ are real and continuous on $\overline{\Omega}$ and
the matrices $a_{ij}$ and $A_{ij}$ are symmetric and positive definite in $\Omega.$

In 1955, Hartman and Wintner\,\cite{har} extended Sturm--Picone theorem (Theorem\,\ref{th.2}) to \eqref{el1}--\eqref{el2} and their theorem reads as follows:

\begin{thm}\label{th.5}
Let $\Omega\subset \R^{n}$ be a bounded domain whose boundary has a piecewise continuous unit normal. 
Suppose $a_{ij}- A_{ij}$ is positive semidefinite and $C\geq c$ on $\overline{\Omega}.$ 
If there exists a nontrivial solution $u$ of $lu=0$ in $\Omega$ such that $u=0$ on $\partial \Omega,$ then every solution of $Lv=0$ vanishes at some point of $\overline{\Omega}.$
\end{thm}

In 1965, Clark and Swanson \cite{cla} obtained a analog of Leighton's theorem (Theorem\,\ref{th.3}) using the variation of $lu,$ which is defined as

$$ V(u)=  \int_{\Omega}  \left[\sum_{i, j=1}^{n} (a_{ij}- A_{ij}) D_{i}u D_{j} u + (C-c) u^{2} \right] dx.    $$

Their theorem reads as follows:

\begin{thm}\label{th.6}
Let $\Omega\subset \R^{n}$ be a bounded domain whose boundary has a piecewise continuous unit normal. Suppose $a_{ij}- A_{ij}$ is positive semidefinite and 
$C\geq c$ on $\overline{\Omega}.$ 
If there exists a nontrivial solution $u$ of $lu=0$ in $\Omega$ such that $u=0$ on $\partial \Omega$ and $V(u)\geq 0,$ then every solution of $Lv=0$ vanishes at some point of
$\overline{\Omega}.$
\end{thm}
Again, it is easy to see that Theorem\,\ref{th.5} is a special case of Theorem\,\ref{th.6} and the proof of Theorem\,\ref{th.6} depends on the following $n$-dimensional version of 
Lemma\,\ref{th.4}. Let us define the quadratic functional associated with \eqref{el2}:

$$M(u)=  \int_{\Omega}  \left[\sum_{i, j=1}^{n} (A_{ij} D_{i}u D_{j} u - C u^{2} \right] dx,    $$
where the domain $\mathcal{D}$ of $M$ is defined to be the set of all real-valued continuous functions on $\overline{\Omega}$ which vanish on the 
boundary and have uniformly continuous first partial derivatives in $\Omega.$

\begin{lem}\rm{($n$-dimensional version of Leighton's variational lemma)}\label{th.7}
If there exists $u \in \mathcal{D}$ not identically zero such that $M(u)\leq 0,$ then every solution $v$ of $Lv=0$ vanishes at some point of $\overline{\Omega}.$
\end{lem}

In recent years, there have been a good amount of research works on the fractional Laplace equations dealing with existence, multiplicity and regularity questions, see for instance 
\cite{aut,dica,di, jtgd,jtgdrbv,fr, franz,lin,pala, ros,seru,ser,tyagij} and many other papers but
to the best of our knowledge, there are not many results available which deal with the qualitative behavior of the solutions such as Sturm-Picone theorem. 
We refer to a very recent paper \cite{sdi} which deals with qualitative behaviours of fractional equations.

Very recently, an attempt is also made to generalize 
the Leighton's variational lemma for a class of fractional Laplace equations. More precisely, the following lemma is establised in \cite{jt}.
\begin{lem}\cite{jt}
Let $2s<n<4s,\,\,0<s<1.$ Let \,$a\in L^{\infty}(\Omega).$ If there exists a function $u\in X_0$ not identically zero such that $j(u)\leq 0,$ then every solution $v$ of
\begin{equation}\label{54}
(-\Delta)^{s} v = a(x) v\,\,\mbox{in}\,\, \Omega;\,\, v= 0\,\,\,\mbox{in}\,\,\R^n \backslash\Omega,\,
\end{equation}
except a constant multiple of $u$ vanishes at some point of $\Omega,$ 
 where
 $$    j(u)=  \int_{\R^n}\left[|(-\Delta)^{\frac{s}{2}}u|^{2} - a(x) u^2\right] dx.$$
 \end{lem}
In the above works, we have defined the fractional Laplacian of $u$ in P.V. integral sense, see Section\,3\,\cite{di} for the details and 
$$  X_0 = \{g\in X : g = 0 \,\,\mbox{a.e.\,\,in}\,\, \R^n \backslash\Omega\},                    $$
where $X$ denotes the linear space of Lebesgue measurable functions
from $\R^n$  to $\R $ such that the restriction to $\Omega$ of any function $g$ in $X$ belongs to $L^2 (\Omega)$ and
the map 
$$ (x,y):\longrightarrow    \frac{g(x)-g(y)}{|x-y|^{\frac{n}{2} +s}}  \in L^{2} (\R^{2n} \backslash(\mathcal{C}\,\Omega \times \mathcal{C}\,\Omega )),         $$
where $\mathcal{C}\,\Omega= \R^n\backslash \Omega.$

Since the classical proof of Sturm-Picone theorem for a pair of ordinary differential equations/ systems as well as elliptic partial differential
equations steadily rests on Leighton's lemma, so it has been extended in different directions and applied to establish oscillation as well as nonoscillation theorems to
differential equations. There are several interesting papers on this subject but for the sake of brevity, we list a few works.
For instance, see the works of 
Jaro\v{s} et\,al.\,\cite{jaros},\, Komkov\,\cite{Komko}, Do\v{s}l\'{y} and \,Jaro\v{s}\,\cite{dosly}, see \cite{tyagi} for a 
generalization of Leighton's variational lemma for nonlinear differential equations and the earlier developments on this area. 
The results of the author\,\cite{tyagi} are used and extended to more general equations by A.\,Tiryaki\,\cite{tiryaki, tir} and in his other papers. 
For a Sturmian comparison and oscillation theorems for a class of half-linear elliptic equations, we refer to \cite{yoshi} and the references cited therein.
Motivated by the above research works and by an increasing interest on fractional Laplace equations and related existence and qualitative questions in recent years, 
it is natural to ask the following question:

\textit{Is there any generalization of Sturm--Picone theorem for a pair of fractional nonlocal equations \eqref{no1}, \eqref{no2}}?

In this paper, we answer the above question affirmatively. 
More precisely, we establish a generalization of  Sturm--Picone theorem for a pair of equations \eqref{no1} and \eqref{no2}. Firstly, we  obtain
Leighton's variational lemma for fractional nonlocal equations by defining the suitable quadratic functional associated with the equation
and then using Leighton's variational lemma, we establish the generalization of Sturm--Picone theorem.

The plan of this paper is as follows. Section 2 deals with the briefs on the fractional nonlocal equations. In Section 3, we state and prove Leighton's variational lemma and establish 
a generalization of Sturm--Picone theorem to fractional nonlocal equations. A few remarks are a part of Section 4.
\section{Fractional nonlocal equations}
Let us recall the very useful briefs on fractional nonlocal equations, see \cite{cast,stin} for the details.

We consider the following fractional nonlocal equation
\begin{eqnarray}\label{pr}
 \begin{gathered}
 (-div. (A(x)\nabla))^{s} u =  C(x) u \,\,\,\mbox{in}\,\,\Omega,\\
 u   = 0 \,\,\,\,\mbox{on}\,\,\,\,\,\,\,\partial \Omega,
\end{gathered}
 \end{eqnarray}
where $\Omega\subset \R^n,\,0<s<1$ is an open bounded subset with smooth boundary, $A$ is real symmetric and positive definite matrix with continuous entries on
$\Omega,\,C\in C(\overline{\Omega}).$
By using the $L^2$-Dirichlet eigenvalues and eigenfunctions $(\lambda_k,  \phi_{k})_{k=0}^{\infty},\,\,\phi_{k}\in H_{0}^{1}(\Omega)$ of $L= (-div. (A(x)\nabla))^{s},$
we can define the fractional powers $L^{s}u = (-div. (A(x)\nabla))^{s} u,\,\,0<s<1,$ for $u$  in the domain Dom$(L^s)\equiv \mathcal{H}^{s}$ in a natural manner, where

\[
 \mathcal{H}^{s}= \begin{cases}
H^{s}(\Omega), \quad &\text{when}~ 0<s<\frac{1}{2},\\
H^{\frac{1}{2}}_{00}(\Omega),  \quad &\text{when}~s= \frac{1}{2},\\
H^{s}_{0}(\Omega), \quad &\text{when}~ \frac{1}{2}<s<1.
\end{cases}
\]
The spaces $H^{s}(\Omega)$ and $H^{s}_{0}(\Omega),\,s\neq \frac{1}{2}$ are the classical fractional Sobolev spaces which are given by
$\overline{C_{c}^{\infty}(\Omega)}$ under the norm
$$ ||u||^{2}_{H^{s}(\Omega)} = ||u||^{2}_{L^{2}(\Omega)} + [u]_{H^{s}(\Omega)}^{2}, $$
 where 
 $$[u]_{H^{s}(\Omega)}^{2}= \int_{\Omega}\int_{\Omega}  \frac{(u(x)- u(y))^2}{|x-y|^{n+2s}} dx dy,\,\,\,\,0<s<1.            $$
The space $H^{\frac{1}{2}}_{00}(\Omega)$ is called Lions-Magenes space which is defined as follows:
$$ H^{\frac{1}{2}}_{00}(\Omega):=\left\{u\in L^{2}(\Omega)|\,\,\,[u]_{H^{\frac{1}{2}}(\Omega)} <\infty,\,\,\text{and}\,\,\int_{\Omega} \frac{(u(x))^2}{dist(x, \partial\Omega)}<\infty
\right\}, $$
see Chapter 1\,\cite{lion} and \cite{noch} for the details.
 Following \cite{cast}, if $u(x)=\sum\limits_{k=0}^{\infty} u_{k}\phi_{k}(x),\,\,x\in \Omega, $ then
$$ L^{s}u(x)=  \sum\limits_{k=0}^{\infty} \lambda_{k}^{s}u_{k}\phi_{k}(x).            $$
One can see that $u=0$ on $\partial \Omega$ and equivalently, we have the semigroup formula
\begin{equation}\label{fr1}
 L^{s}u(x)= \frac{1}{\Gamma (-s)} \int_{0}^{\infty} (e^{-t L}u(x) - u(x)    ) \frac{dt}{t^{1+s}},
 \end{equation}
where $\{ e^{-t L}u \}_{t>0}$ is the heat diffusion semigroup generated by $L$ with the Dirichlet boundary conditions and $\Gamma$ is the Gamma function.
Now, as it is already known that (see\,\cite{stin}) the fractional operators \eqref{fr1} can be described as Dirichlet-to-Neumann maps for an extension problem in the spirit of 
the extension problem for the fractional Laplacian on $\R^n$ of \cite{caff}. More precisely, let $U=U(x, y):\,\Omega \times (0, \infty) \longrightarrow \R$ be the solution of the
following degenerate elliptic equation with $A_2$ weight:

\begin{equation}\label{fr2}
\begin{aligned}
 div(y^{a} B(x)\nabla U)&  =  0 \,\,\,\mbox{in}\,\,\Omega\times (0, \infty),\\
 U   &= 0 \,\,\,\,\mbox{on}\,\,\,\partial \Omega \times[0, \infty),\\
 U(x, 0)& = u(x)\,\,\,\,\mbox{on}\,\,\,\Omega,
 \end{aligned}
 \end{equation}
 where
 
 \begin{equation}\label{fr3}
B(x):=
\begin{bmatrix}
A(x) &  0\\
0 & 1 
\end{bmatrix}
\in \R^{n+1}\times \R^{n+1}\,\,\mbox{and}\,\,a:= 1-2s \in (-1, 1).
\end{equation}
 Then we have
 
 $$ - \frac{1}{2s} \lim\limits_{y\rightarrow 0^+} y^{a} U_{y}(x, y) = - \lim\limits_{y\rightarrow 0^+} \frac{U(x, y)- U(x, 0)}  {y^{2s}}= c_{s} L^{s} u(x),\,\,x\in \Omega.                                     $$
 Also, there are explicit formulas for $U$ in terms of the semigroup $e^{-tL},$ see Theorem 2.5 \,\cite{cast}.

 In view of \eqref{fr1} and \eqref{fr2}, Equation\,\eqref{pr} turns out the following
 
 \begin{equation}\label{fr4}
\begin{aligned}
 div(y^{a} B(x)\nabla U)&  =  0 \,\,\,\mbox{in}\,\,\Omega\times (0, \infty),\\
 U   &= 0 \,\,\,\,\mbox{on}\,\,\,\partial \Omega \times[0, \infty),\\
- \frac{1}{2s}  \lim\limits_{y\rightarrow 0^+} y^{a} U_{y}(x, y) & = c_{s} C(x) u\,\,\,\,\mbox{on}\,\,\,\Omega,
 \end{aligned}
 \end{equation}
where $U$ and $u$ are related by $U(x, 0) = u(x)\,\mbox{in}\,\Omega$ and $B$ is defined in \eqref{fr3}.

Using Theorem\,2.5\,\cite{cast}, one can see the existence of a unique weak solution to \eqref{fr4}, which is given below:\\
Let $u\in \mathcal{H}^{s},\,C\in L^{\infty}(\Omega),$
then there exists a unique weak solution  $U\in H_{0}^{1}(\Omega \times (0, \infty), y^{a} dX)$ of \eqref{fr4}, where $B$ and $a$ are defined as above. More precisely, for each
$\phi \in H_{0}^{1}(\Omega \times (0, \infty), y^{a} dX),$
\begin{equation}
  \int_{\Omega} \int_{0}^{\infty} y^{a} B(x) \nabla U\,\nabla \phi dX= c_{s}\int_{\Omega}C(x) u(x) \phi(x, 0) dx.
\end{equation}

Let us recall the following boundary regularity of the solution of \eqref{pr}, see Theorem\,1.5\,\cite{cast} for the details.

\rm{(i)} Suppose that $2(2s-1)^{+} <n<4s,$ and $\Omega$ is a $C^{1}$ domain and that $A(x)$ is continuous in $\overline{\Omega}.$ Then $u\in C^{0, \alpha} (\overline{\Omega}),$
for $\alpha= 2s -\frac{n}{2}.$

\rm{(ii)} Suppose that $s>\frac{1}{2},\,n< 2 (2s -1) ,\,\,\Omega$ is a $C^{1, \alpha}(\Omega)$ domain and $A(x)$ is in $C^{0, \alpha}(\Omega),$ for $\alpha = 2s -n \in (0,1).$
Then $u\in C^{1, \alpha}(\overline{\Omega}).$
 
 \section{Sturm-Picone Theorem}
  In this section, we state and prove Sturm-Picone theorem. 
  
  Let 
$\mathcal{D}:= \{U\in H_{0}^{1}(\Omega \times (0, \infty), y^{a} dX): U(x, 0)= u(x)      \}$
and let  us define the quadratic functional assocaied with \eqref{fr4}:
\begin{equation}\label{fl}
  M(U)= \int_{\Omega}\int_{0}^{\infty}\sum\limits_{i, j} y^{a} B_{i j}D_{i}U D_{j}U dX- 2sc_{s}\int_{\Omega} C(x) (U(x, 0))^{2} dx,\,\,\,\,U\in \mathcal{D}.  
 \end{equation}
The following is an important lemma to establish Sturm-Picone theorem.
 \begin{lem}(Leighton's Variational Lemma)\label{lto}
  Let $s,\,\Omega$ and $A$ be defined as in \rm{(i)} and \rm{(ii)}.  If there exists $U\in \mathcal{D} $ not identically zero such that $M(U)\leq  0,$ then every solution $v$ of \eqref{fr4} vanishes at some point of
  $\overline{\Omega}\times[0, \infty).$ Also, if $C \in C(\overline{\Omega}),\,C>0,$ then every nontrivial solution of \eqref{pr} vanishes at some point of $\Omega.$
  \end{lem}

 \begin{proof}
  We will prove this lemma by the method of contradiction. Suppose there exists a solution $v$ of \eqref{fr4} such that $v\neq 0$ on $\overline{\Omega}\times [0, \infty).$ 
  For $U\in \mathcal{D}, $
  let us define
  $$ X^{i} =  v D_{i}\left( \frac{U}{v} \right), \,\,\,\,\,\,Y^{i}= \frac{1}{v} \sum\limits_{ j}  y^{a} B_{ij} D_{j}v,\,\,\,\,i=1,2,\cdots,n.        $$
  $$  X^{n+1} =  v D_{y}\left( \frac{U}{v} \right),\,\,\,\,\,\,Y^{n+1}= \frac{1}{v} \sum\limits_{ j}  y^{a} B_{(n+1)j} D_{j}v =\frac{1}{v}   y^{a}  D_{y}v.     $$
  $$ G(U, v)=  \sum\limits_{ i,j}  y^{a} B_{ij} X^{i}X^{j}  +  \sum\limits_{ i} D_{i} (U^{2}Y^{i}).               $$
  Now, one can establish the following identity in $\Omega\times [0, \infty):$
 \begin{equation}\label{id1}
  G(U, v)= \sum\limits_{ i,j}  y^{a} B_{ij} D_{i}U\,D_{j}U + \frac{U^2}{v} Lv,
 \end{equation}
where\,\,$Lv=  div(y^{a} B(x)\nabla v).$ Indeed, 
 
 \begin{equation} \label{idd}
\begin{aligned}
 \sum\limits_{ i,j}  y^{a} B_{ij} X^{i}X^{j}  +  \sum\limits_{ i} D_{i} (U^{2}Y^{i})& = \frac{1}{v^2}  \sum\limits_{ i,j}  y^{a} B_{ij}(vD_{i}U- UD_{i}v)(vD_{j}U- UD_{j}v)   \\                       
& + \frac{2U}{v} \sum\limits_{ i,j}  y^{a} B_{ij}  D_{i}U\,D_{j}v + \frac{U^2}{v^2}  \sum\limits_{ i,j} \left(v D_{i} (y^{a} B_{ij} D_{j}v) - y^{a}  B_{ij} D_{i}v  D_{j}v \right ).
\end{aligned}
\end{equation}
 Since $B_{ij}$ is symmetric, so \eqref{idd} reduces to the RHS of \eqref{id1}. Now, since from \eqref{fr4}, $Lv=0$ in $\Omega\times (0, \infty),$  so from \eqref{fl} and 
 \eqref{id1}, it follows that
 \begin{equation}\label{po}
  \int_{\Omega}\int_{0}^{\infty} \sum\limits_{i, j} y^{a}B_{i j}D_{i}U D_{j}U dX =  \int_{\Omega}\int_{0}^{\infty}\left[ \sum\limits_{i, j} y^{a} B_{i j} X^{i}X^{j} +
  \sum\limits_{ i} D_{i} (U^{2}Y^{i})\right] dX.  
 \end{equation}
Since $U$ vanishes on $\partial \Omega \times[0, \infty),$ so by Green's theorem and third equation of \eqref{fr4}, we get
\begin{equation}\label{mag}
 \int_{\Omega}\int_{0}^{\infty} \sum\limits_{ i} D_{i} (U^{2}Y^{i})) dX= 2s c_{s} \int_{\Omega} C(x) (U(x, 0))^{2} dx.
\end{equation}
Now, \eqref{po} and \eqref{mag} yields that

\begin{equation}\label{maza}
 \int_{\Omega}\int_{0}^{\infty}\sum\limits_{i, j} y^{a} B_{i j}D_{i}U D_{j}U dX- 2sc_{s}\int_{\Omega} C(x) (U(x, 0))^{2} dx= 
 \int_{\Omega}\int_{0}^{\infty} \sum\limits_{i, j} y^{a} B_{i j} X^{i}X^{j}.
\end{equation}
Since $(B_{ij})$ is positive definite so from \eqref{fl} and \eqref{maza}, we get $M(U)\geq 0,$ and equality holds if and only if
 $X^{i}\equiv 0$ for each $i= 1, 2, 3,\,\cdots,\,n,\,n+1,$  i.e., $U$  is a constant multiple of $v.$ But this is not possible, since $U=0$ on $\partial \Omega \times[0, \infty)$ 
 while $v\neq 0$  on $\overline{\Omega}\times [0, \infty),$ and therefore $M(U)>0,$ which is a contradiction. This implies that there exists 
 $(x, y_1) \in \overline{\Omega}\times[0, \infty)$ such that $v(x, y_1)=0.$
 
 Now, by using the heat kernel, we can see that the solution of \eqref{fr4} can be represented in terms of Poisson's kernel, i.e.,
 \begin{equation}\label{vv1}
 v(x, y)= c_{s} \int_{\Omega} P^{s}_{y} (x, z)  C(z) u(z) dz,
 \end{equation}
 where  $u$ satisfies Equation\,\eqref{pr} and $P^{s}_{y}(x, z)$ is the Poisson kernal, which is given by
\begin{equation}\label{vv2}
 P^{s}_{y} (x, z)= \frac{y^{2s}}{4^{s}  \Gamma (s)} \int_{0}^{\infty}e^{- \frac{y^2}{4t}} W_{t}(x, z) \frac{dt}{t^{1+s}},
 \end{equation}
see, pp.\,777\,\cite{cast} for the details. In \eqref{vv2}, $W_{t}(x, z)$ is the distributional heat kernel for $ L= -div. (B(x)\nabla) $  with the Dirichlet boundary condition, which is given by
\begin{equation}\label{vv3}
 W_{t}(x, z)= \sum\limits_{k=0}^{\infty} e^{-t \lambda_{k}} \phi_{k}(x) \phi_{k} (z)  = W_{t}(z, x),\,\,\,t>0,\,x, z\in \Omega.
 \end{equation}
From \cite{dav}, it is clear that $W_{t}(x, z)>0,\,\,\forall\,\,t>0$ and $x, z\in \Omega.$
From \eqref{vv1}, \eqref{vv2} and \eqref{vv3}, it follows that $v(x, y_1)=0$ implies that either $y_1=0$ or $u$ changes sign in $\Omega$ and in both cases, $u$
vanishes at some point of $\Omega.$ This completes the proof of the lemma.
 \end{proof}

 \begin{rem}
  It will be of interest to remove the sign condition on $C$ for the second part of the lemma.
 \end{rem}

Let us consider a pair of nonlocal equations:

\begin{eqnarray}\label{eqn1}
 \begin{gathered}
 (-div. (A_1(x)\nabla))^{s} u =  C_{1}(x) u \,\,\,\mbox{in}\,\,\Omega,\\
 u   = 0 \,\,\,\,\mbox{on}\,\,\,\,\,\,\,\partial \Omega,
\end{gathered}
 \end{eqnarray}
and
 \begin{eqnarray}\label{eqn2}
 \begin{gathered}
 (-div. (A_2(x)\nabla))^{s} w =  C_{2}(x) w \,\,\,\mbox{in}\,\,\Omega,\\
 w   = 0 \,\,\,\,\mbox{on}\,\,\,\,\,\,\,\partial \Omega,
\end{gathered}
 \end{eqnarray}
where  $\Omega\subset \R^n$ is an open bounded subset with smooth boundary, $0<s<1,\,A_1,\,A_2$ are real symmetric and positive definite matrices
on $\Omega$ with continuous entries on $\overline{\Omega}$ and $C_{1}, C_{2}\in C(\overline{\Omega}).$ 

In view of \eqref{fr1} and \eqref{fr2}, Equations\,\eqref{eqn1} and \eqref{eqn2} turn out the following
 
 \begin{equation}\label{no3}
\begin{aligned}
 div(y^{a} B_1(x)\nabla U)&  =  0 \,\,\,\mbox{in}\,\,\Omega\times (0, \infty),\\
 U   &= 0 \,\,\,\,\mbox{on}\,\,\,\partial \Omega \times[0, \infty),\\
- \frac{1}{2s}  \lim\limits_{y\rightarrow 0^+} y^{a} U_{y}(x, y) & = c_{s} C_{1}(x) u\,\,\,\,\mbox{on}\,\,\,\Omega,
 \end{aligned}
 \end{equation}

and
\begin{equation}\label{no4}
\begin{aligned}
 div(y^{a} B_2(x)\nabla W)&  =  0 \,\,\,\mbox{in}\,\,\Omega\times (0, \infty),\\
 W   &= 0 \,\,\,\,\mbox{on}\,\,\,\partial \Omega \times[0, \infty),\\
- \frac{1}{2s}  \lim\limits_{y\rightarrow 0^+} y^{a} W_{y}(x, y) & = c_{s} C_{2}(x) w\,\,\,\,\mbox{on}\,\,\,\Omega,
 \end{aligned}
 \end{equation}
respectively, where  $U(x, 0) = u(x),\,\,W(x, 0)= w(x)\,\,\mbox{in}\,\Omega$ and $B_1$ and $B_2$ are defined as follows: 

\begin{equation}\label{frqn}
B_1(x):=
\begin{bmatrix}
A_1(x) &  0\\
0 & 1 
\end{bmatrix},\,\,
B_2(x):=
\begin{bmatrix}
A_2(x) &  0\\
0 & 1 
\end{bmatrix}
\in \R^{n+1}\times \R^{n+1}.
\end{equation}
Let us define the quadratic functionals associated with \eqref{no3} and \eqref{no4}, respectively:
\begin{equation}\label{fln1}
M_{1}(U)= \int_{\Omega}\int_0^\infty \sum_{i, j} y^{a} B_{{1}_{{i, j}}} D_{i}U D_{j}U dX- 2sc_{s}\int_{\Omega} C_{1}(x) (U(x, 0))^{2} dx,\,\,\,\,\,U\in \mathcal{D}. 
 \end{equation}

\begin{equation}\label{fln2}
  M_{2}(U)= \int_{\Omega}\int_{0}^{\infty} \sum_{i, j} y^{a} B_{{2}_{{i, j}}}D_{i}U D_{j}U dX- 2sc_{s}\int_{\Omega} C_{2}(x) (U(x, 0))^{2} dx,\,\,\,\,\,U\in \mathcal{D}  
 \end{equation}
and the variation is given by 
$$V(U)= M_{2}(U)- M_{1}(U)$$
$$\,\,\,\,\,\,\,\,\,\,\,=\int_{\Omega}\int_{0}^{\infty}\sum_{i, j} y^{a} (B_{{2}_{{i, j}}}-  B_{{1}_{{i, j}}}) D_{i}U D_{j}U dX +
2sc_{s}\int_{\Omega} (C_{1}(x) - C_{2}(x)) (U(x, 0))^{2} dx,\,\,\,\,U\in \mathcal{D},\,y>0.$$

\begin{thm} (Generalization of Leighton's Theorem)\label{lego}
 Let $s,\,\Omega$ be defined as in \rm{(i)} and \rm{(ii)}. Let $A_1,\,A_2$ be real symmetric and positive definite matrices
on $\Omega$ which satisfy \rm{(i)} and \rm{(ii)} and $C_{1}, C_{2}\in C(\overline{\Omega}).$ 
 Let $U$ be  nontrivial solution of  \eqref{no4} such that $V(U)\geq 0,\,$ then every solution of \eqref{no3} vanishes at some point of $\overline{\Omega}\times [0, \infty).$
 In addition, if $C_{1}>0$ in $\overline{\Omega},$ then every nontrivial solution of \eqref{eqn1} vanishes at some point of $\Omega.$
\end{thm}
\begin{proof}
 Since $U$ is a nontrivial solution of \eqref{no4}, so, Green's formula yields that
 $M_{2}(U)=0.$ Since $V(U)= M_{2}(U)- M_{1}(U)\geq 0,$ i.e., $M_{1}(U)\leq M_{2}(U)=0,\,\,U\in \mathcal{D}.$ Now, by an application of Lemma\,\ref{lto}, 
 every solution of \eqref{no3}  vanishes at 
 some point of $\overline{\Omega}\times [0, \infty).$  Also, every nontrivial solution of \eqref{eqn1} vanishes at some point of $\Omega.$ 
 This completes the proof.
 \end{proof}
\begin{thm} (Sturm--Picone Comparison Theorem)\label{stm}
 Let $s,\,\Omega$  be defined as in \rm{(i)} and \rm{(ii)}. Let $A_1,\,A_2$ be real symmetric and positive definite matrices
on $\Omega$ which satisfy \rm{(i)} and \rm{(ii)} and let
 $C_{1},  C_{2}\in C(\overline{\Omega})$ with $C_{1}(x)- C_{2}(x)\geq 0 $ on $\overline{\Omega}.$  
Let $B_{{2}_{ij}}$--$B_{{1}_{ij}}$ be positive semidefinite and $U$ be  nontrivial solution of \eqref{no4}, then every solution of \eqref{no3} vanishes at some point 
of $\overline{\Omega}\times [0, \infty).$ In addition, if $C_{1}>0$ in $\overline{\Omega},$ then every nontrivial solution of \eqref{eqn1} vanishes at some point of $\Omega.$
\end{thm}
\begin{proof}
Since $B_{{2}_{ij}}$--$ B_{{1}_{ij}}$ is positive semidefinite and $C_{1}(x)- C_{2}(x)\geq 0 $ on $\overline{\Omega},$  so $V(U)\geq 0$ and the proof follows from Theorem\,\ref{lego}.

\end{proof}

\section{A few remarks }
 A few remarks concerning the qualitative behavior of the solution to fractional Laplace equations are in order:
 
 \begin{rem}
  Let $\Omega(r_0)= \{x\in \R^{n}:\,\,||x|| \geq r_0    \}$ for some $r_0>0$ be an exterior domain in $\R^n,$ where $||\,\,\cdotp\,||$ is the usual Euclidean norm in $\R^n.$
  Now, a very first question concerns  whether one can pose \eqref{pr} in exterior domains. If yes, then a
   nontrivial solution $u$ of \eqref{pr} (posed in exterior domains) is said to be oscillatory if the set $\{x\in \Omega(r_0):\,\,u(x)=0  \}$ is unbounded; otherwise it is 
   called non-oscillatory,
  see for instance, pp.\,135\,\cite{xu}. Equation \eqref{pr} is called oscillatory if all its solutions are oscillatory. In this context, 
  it is natural to look at the whole study of this paper in the exterior/unbounded domains. 
 \end{rem}

 The next remark deals with an evidence of the oscillatory behavior of the solution of \eqref{pr} in $\R.$
 \begin{rem}
 Let us consider \eqref{pr} in $\R$ with $A(x)= 1$ and $C=4,\,\,s\in (0, 1)$ and consider the radially symmetric solutions of 
 \begin{equation}\label{r1}
 -(-\Delta)^{s} u + 4 u=0\,\,\text{in}\,\,\,\,\R.
 \end{equation}
 We recall that, when $u$ is radially symmetric, 
   \begin{equation}\label{rad}
   -(-\Delta)^{s} u= u''(r) + (n+1-2s) \frac{u'(r)}{r}, \,\,u(x)=u(|x|)= u(r),\,\,\,x\in \R^n,          
 \end{equation}
 see \cite{fer} for the details and by \eqref{rad}, \eqref{r1} converts into

\begin{equation}\label{r2}
u''(r) + 2(1-s) \frac{u'(r)}{r}+ 4 u = 0.
\end{equation}
Now, using the standard transformation 
$$ u(r)= y(r) e^{-\frac{1}{2} \int \frac{2 (1-s)}{r}dr  }= y(r) r^{s-1},\,\,r>0, $$
\eqref{r2} is

\begin{equation}\label{r3}
 y''(r) + \left(\frac{4 r^2 - s^2 +s}{r^2}\right) y(r)=0.
\end{equation}
Since $s\in (0, 1),$ so $\frac{4 r^2 - s^2 +s}{r^2} >4$ and by classical Sturm's comparison theorem, \eqref{r3} is oscillatory. Since the above transformation is 
oscillation preserving, and therefore \eqref{r2} is oscillatory. Now, it will be of interest to find out  whether every solution of \eqref{r1} and more generally, to \eqref{pr},
when posed in $\R^n,$ is oscillatory.

 \end{rem}
 
 In the next remark, one can also inquire on the non-oscillatory solution (eventually one signed solution) of \eqref{pr} in $\R.$
 \begin{rem}
 Let us consider \eqref{pr} in $\R$ with $A(x)= 1$ and $C=-1,\,\,s\in (0, 1)$ and consider the radially symmetric solutions  of 
 \begin{equation}\label{rnon}
 -(-\Delta)^{s} u -  u=0\,\,\text{in}\,\,\,\,\R.
 \end{equation}
 Again by \eqref{rad}, \eqref{rnon} converts into

\begin{equation}\label{rn2}
u''(r) + 2(1-s) \frac{u'(r)}{r}-  u = 0.
\end{equation}
Now, using the standard transformation 
$$ u(r)= y(r) e^{-\frac{1}{2} \int \frac{2 (1-s)}{r}dr  }= y(r) r^{s-1},\,\,r>0, $$
\eqref{rn2} reads as
\begin{equation}\label{rn3}
 y''(r) + \left(\frac{-r^2 - s^2 +s}{r^2}\right) y(r)=0.
\end{equation}
It is easy to see that  $\frac{1}{4 r^2} > \frac{-r^2 - s^2 +s}{r^2}$ and we know that 
\begin{equation}\label{rno}
 y''(r) + \frac{1}{4r^2} y(r)=0
\end{equation}
is non-oscillatory so by classical Sturm's comparison theorem, \eqref{rn3} is non-oscillatory and 
therefore \eqref{rn2} is non-oscillatory. Now, it will be of interest to investigate whether  \eqref{rnon} and more generally, \eqref{pr},
when posed in $\R^n,$ is non-oscillatory.
 \end{rem}
 

\vskip .002in

\begin{center}
{\bf Acknowledgment} \end{center}
Author thanks DST/SERB for the financial support under the grant EMR/2015/001908.


\begin{thebibliography}{0}



\bibitem{aut} G. Autuori,  A.Fiscella,  P. Pucci, 
Stationary Kirchhoff problems involving a fractional elliptic operator and a critical nonlinearity, Nonlinear Anal. \textbf{125}, 2015, 699--714. 

\bibitem{caff} L. Caffarelli, L. Silvestre, An extension problem related to the fractional Laplacian, Commun. Partial Differ. Equ. \textbf{32}, 2007, 1245--1260. 

 \bibitem{cast} L.A. Caffarelli,  P. R. Stinga,  Fractional elliptic equations, Caccioppoli estimates and regularity,  Ann. Inst. H. Poincar\'{e} Anal. Non Lin\'{e}aire \textbf{33} (2016), no. 
 3, 767--807.


\bibitem{cla} C.Clark and C.A.Swanson, Comparison theorems for elliptic differential equations, Proc. Amer. Math. Soc., \textbf{16}, 1965, 886--890.


\bibitem{dav} E.B. Davies, Heat Kernels and Spectral Theory, Cambridge Tracts in Mathematics, Vol. 92, Cambridge University Press, Cambridge, 1989.



\bibitem{dica} A. Di Castro, T. Kuusi, G.\,Palatucci, Nonlocal Harnack inequalities,
J. Funct. Anal. \textbf{267}, 2014, no. 6, 1807--1836. 


\bibitem{di}  E. Di Nezza, G. Palatucci, E. Valdinoci, Hitchhiker's guide to the fractional Sobolev spaces, Bull. Sci. Math. \textbf{136}, 2012, no. 5, 521--573.

\bibitem{sdi} S. Dipierro, O. Savin, E. Valdinoci, 
All functions are locally $s$-harmonic up to a small error, J. Eur. Math. Soc. (JEMS) \textbf{19}, 2017, no. 4, 957--966. 


\bibitem{dosly} O.\,Do\v{s}l\'{y}, J.\,Jaro\v{s}, A singular version of Leighton's comparison theorem for forced quasilinear second-order 
differential equations, Arch.\,Math. (BRNO)\,\textbf{39}, 2003, 335--345.


\bibitem{jtgd} G. Dwivedi, J. Tyagi, and R. B. Verma, On the bifurcation results for fractional Laplace equations,  Math. Nach.,  \textbf{290}, 2017, no. 16, 2597--2611. 

\bibitem{jtgdrbv} G. Dwivedi, J. Tyagi, and R. B. Verma, Stability of positive solution to fractional logistic equations, to appear in Funkcialaj Ekvacioj.


 \bibitem{fr} A. Farina, B. Sciunzi, E. Valdinoci, Bernstein and De Giorgi type problems: new
results via a geometric approach, Ann. Sc. Norm. Super. Pisa Cl. Sci.,\,\textbf{5} 7(4), 2008, 741--791.

 \bibitem{fer} F. Ferrari, I.E.Verbitsky, Radial fractional Laplace operators and Hessian inequalities, J. Differential Equations 253 (2012), no. 1, 244--272.

 \bibitem{franz} G. Franzina, G. Palatucci, Fractional $p$-eigenvalues, Riv. Mat. Univ. Parma, (N.S.) \textbf{5}, 2014, no. 2, 373--386.


 \bibitem{ghat}  A. Ghatasheh, R. Weikard, On Leighton's comparison theorem, J. Differential Equations \textbf{262} (2017), no. 12, 5978--5989. 


\bibitem{har} P. Hartman, A. Wintner, On a comparison theorem for self-adjoint partial differential equations of elliptic type, Proc. Amer. Math. Soc., \textbf{6}, 1955, 862--865.

\bibitem{jaros} J.\,Jaro\v{s}, T.\,Kusano and N. Yoshida, Forced superlinear oscillations via Picone's identity, Acta Math. 
\,Univ.\,Comenianae \textbf{LXIX}, 2000,\,107--113.


\bibitem{Komko} V.\,Komkov,\,A generalization of Leighton's variational theorem, Applicable Analysis \textbf{2}, 1972,
377--383.


\bibitem{leig} W.\,Leighton, Comparison theorems for linear
differential equations of second order, Proc.\,Amer.\,Math.\,Soc.\,\textbf{13}, 1962, 603--610.

\bibitem{lin} E. Lindgren,  P. Lindqvist,  Fractional eigenvalues, Calc. Var. Partial Differential Equations \textbf{49}, 2014, no. 1-2, 795--826. 

\bibitem{lion} J.-L. Lions, E. Magenes, Probl\`{e}mes aux Limites Non Homog\`{e}nes et Applications, vol. 1, Travaux et Recherches Mathématiques, vol. 17, Dunod, Paris, 1968.



\bibitem{noch} R.H. Nochetto, E. Otárola, A.J. Salgado, A PDE approach to fractional diffusion in general domains: a priori error analysis, 
Found. Comput. Math. \textbf{15}, 2015, no. 3, 733--791.  


\bibitem{pala} G.\,Palatucci,  A.\,Pisante, Improved Sobolev embeddings, profile decomposition, and concentration-compactness for fractional Sobolev spaces,
Calc. Var. Partial Differential Equations \textbf{50}, 2014, no. 3-4, 799--829. 

\bibitem{picone} M.\,Picone, Sui valori eccezionali di un parametro da cui dipende un'equazione differenziale lineare ordinaria del second'ordine,
Ann. Scuola Norm. Sup. Pisa Cl. Sci. \textbf{11}, 1910, 144. 

\bibitem{ros} X. Ros-Oton,  J. Serra, The extremal solution for the fractional Laplacian, Calc. Var. Partial Differential Equations\,\textbf{50}, 2014, 723--750.

\bibitem{seru} R. Servadei and E. Valdinoci, Lewy-Stampacchia type estimates for variational inequalities driven by
(non)local operators, Rev. Mat. Iberoam. \textbf{29}, 2013, no. 3, 1091--1126. 


\bibitem{ser} R. Servadei, E. Valdinoci,  Mountain pass solutions for non-local elliptic operators, J. Math. Anal. Appl. \textbf{389}(2), 2012, 887--898.



\bibitem{stin} P.R. Stinga, J.L. Torrea, Extension problem and Harnack's inequality for some fractional operators, Commun. Partial Differ. Equ. \textbf{35},  2010, 2092--2122.


\bibitem{sturm} C.\,Sturm, Sur les \'{e}quations diff\'{e}rentielles lin\'{e}aries du second ordere, 
J. Math.\,Pures.\,Appl.\,\textbf{1}, 1836, 106--186.




 \bibitem{tiryaki} A. Tiryaki,  Sturm-Picone type theorems for second-order nonlinear elliptic differential equations, Electron. J. Differential Equations 2014, No. 214, 10 pp.


 \bibitem{tir} A. Tiryaki, Sturm-Picone type theorems for nonlinear differential systems, Electron. J. Differential Equations 2015, No. 154, 9 pp.
\bibitem{tyagi} J. Tyagi, Generalizations of Sturm-Picone theorem for second-order nonlinear differential equations, Taiwanese J. Math. \textbf{17}, 2013, no.1,  361--378. 
\bibitem{tyagij} J.Tyagi, Eigenvalue problem for fractional Kirchhoff Laplacian,  Rendiconti Lincei Matematica e Applicazioni\, \textbf{29}, 2018, no. 1,\, 195--203.
\bibitem{jt} J.Tyagi,   A note on Leighton's variational lemma for fractional Laplace equations, \,Zeitschrift f\"{u}r Analysis und ihre Anwendungen,
Volume \textbf{37},  2018, no. 4,  461--473.

 
\bibitem{xu} Z.Xu, On second order damped elliptic oscillation, J. Math. Pures Appl. (9) \textbf{89}, 2008, no. 2, 134--144. 
 
\bibitem{yoshi} N.Yoshida, Sturmian comparison and oscillation theorems for a class of half-linear elliptic equations, Nonlinear Anal. \textbf{71} (2009), no. 12, e1354--e1359.
 
 
 
 
 
\end{thebibliography}
\end{document}